\documentclass[11pt]{amsart}
\usepackage[top=3cm, bottom=2.5cm, left=2cm, right=2cm]{geometry}

\usepackage{amsmath,amssymb,amsthm}
\usepackage{mathrsfs}
\usepackage[pdftex]{graphicx}
\usepackage{epsfig}
\usepackage[usenames,dvipsnames]{color}
\usepackage{enumerate}
\usepackage{bookmark}
\usepackage{float}

\newcommand{\suchthat}{\;\ifnum\currentgrouptype=16 \middle\fi|\;}

\usepackage{setspace}
\numberwithin{equation}{section}
\newtheorem{theorem}{Theorem}[section]
\newtheorem{definition}[theorem]{Definition}
\newtheorem{proposition}[theorem]{Proposition}
\newtheorem{lemma}[theorem]{Lemma}
\newtheorem{corollary}[theorem]{Corollary}
\newtheorem{example}[theorem]{Example}
\newtheorem{remark}[theorem]{Remark}

\newcommand{\BN}{\mathbb{N}}

\newcommand{\BC}{\mathbb{C}}

\newcommand{\cS}{\mathcal{S}}

\newcommand{\cG}{\mathcal{G}}		

\newcommand{\cI}{\mathcal{I}}		

\newcommand{\cH}{\mathcal{H}}

\newcommand{\cT}{\mathcal{T}}

\newcommand{\fG}{ \mathfrak{G}}

\newcommand{\ov}{\overline}

\newcommand{\Id}{\mbox{Id}}

\newcommand{\e}{\mathbf{e}}

\newcommand{\z}{\mathbf{z}}
\newcommand{\w}{\mathbf{w}}
\newcommand{\ww}{\omega}

\newcommand{\bnu}{\boldsymbol\nu}
\newcommand{\bmu}{\boldsymbol\mu}
\newcommand{\bta}{\boldsymbol\eta}

\newcommand{\bgamma}{\boldsymbol\gamma}

\newcommand{\llangle}{\langle \kern -0.2em \langle}	
\newcommand{\rrangle}{\rangle \kern -0.2em \rangle}
\newcommand{\inner}[1]{\left\langle  #1 \right\rangle }

\newcommand{\Span}{\mbox{\rm Span }}
\newcommand{\od}{\mbox{\rm ~(mod 3) }}

\newcommand{\soul}[1]{#1_r}
\newcommand{\comp}[1]{\overline{\cG}_3^{(#1)}}

\usepackage{hyperref}

\begin{document}
\title[Generalized Grassmann Algebras]{Generalized Grassmann Algebras and Applications to Stochastic Processes}

\author[D. Alpay, P. Cerejeiras, U. K\"ahler]{D. Alpay$^\dagger$, P. Cerejeiras$^\ddagger$, U. K\"ahler$^\ddagger$}
\address{$^\dagger$ Schmid College of Science and Technology, \newline Chapman University \newline One University Drive \newline Orange, California 92866, USA \vspace{2mm} \newline $^\ddagger$ CIDMA - Center for Research and Development in Mathematics and Applications, \newline Department of Mathematics, University of Aveiro \newline Campus Universit\'ario de Santiago \newline 3810-193 Aveiro, Portugal.}


\date{\today}


\begin{abstract}
In this paper we present the groundwork for an It\^o/Malliavin stochastic calculus and Hida's white noise analysis in the context of a supersymmentry with $\mathbb{Z}_3$-graded algebras. To this end we establish a ternary Fock space and the corresponding strong algebra of stochastic distributions and present its application in the study of stochastic processes in this context.
\end{abstract}

\maketitle

\tableofcontents

\section{Introduction}

Classic theories like Bose-Einstein or Fermi-Dirac statistics are based on SU(2)-symmetries and $\mathbb{Z}_2$-graded algebras. But theories like Quantum Chromodynamics where quarks are considered as fermions require a setting with $\mathbb{Z}_3$- (or $\mathbb{Z}_6$-)graded algebras for a convenient generalisations of Pauli's exclusion principle and the establishment of the corresponding statistics~\cite{Kerner2018}. Such algebras and the corresponding Dirac operators have been studied in the recent past. But these theories raise an additional question about the necessary extension of the corresponding supersymmetry. While standard supersymmetry combines bosonic fields with fermionic fields to a $\mathbb{Z}_2$-graded Lie super algebra a super-symmetric extension involving quarks requires a $\mathbb{Z}_3$-graded algebra of Grassmannian type which leads to a kind of hypersymmetry~\cite{Abramov1997}.

One important aspect of supersymmetry lies in its combination with stochastic dynamics, also called topological supersymmetry initiated by the seminal works of Parisi and Sourlas in 1979 (see e.g. \cite{PaSo1979} or \cite{Junk1996}), in particular the more recent supersymmetric theory of stochastic dynamics~\cite{Ovch11,Ovch16}. But for the establishment of such a theory in the context of $\mathbb{Z}_3$- (or $\mathbb{Z}_6$-)graded algebras, i.e. the development of topological hypersymmetry, a major problem arises. One needs the counterpart of the classic It\^o/Malliavin stochastic calculus and Hida's white noise analysis in this context. This is the problem we are going to study in this paper. While differentiation and integration has been studied in superspace using Clifford algebras~\cite{Ali1,Ali2} no such study has been made in the context of $\mathbb{Z}_3$-graded algebras to the knowledge of the authors.

To this end let us recall the classic (finite or infinite) Grassmann algebra. $\cG_d$ is the unital algebra over the complex numbers generated by $\e_0=1$ and a finite set of elements $\e_i$, $i\in\{1,2,\hdots,d\}$, which do not belong to and are linearly independent over $\mathbb{C}$. Moreover, they satisfy
\begin{equation}
\e_i \e_j + \e_j \e_i=0, \label{anticommutative}
\end{equation}
where $i,j \in\{1,2,\hdots,d\}$, and in particular
\begin{equation}
\e_i^2 = 0. \label{square-0}
\end{equation}
An element of $\cG_d$ is often referred to as a supernumber. The number of generators can be taken to be infinite. To do so, closures of the algebra with respect to a norm are considered -- see, e.g., \cite{rogers1980global, zbMATH00861741, alpay2018distribution, alpay2018positivity}.

It is well-known that the classic Grassmann algebra is $\mathbb{Z}_2$-graded. To introduce a $\mathbb{Z}_3$-grading we present a generalization of the classic Grassmann algebra to a ternary Grassmanian setting. We consider once again the real generator $\e_0=1$ and basis elements $\e_1,\hdots,\e_d$. Instead of \eqref{anticommutative}, we assume the generators satisfy
\begin{equation} \label{Eq.001}
\e_i \e_j = \omega \e_j \e_i
\end{equation}
for every $i<j$, where $i,j\in\{1,2,\hdots,d\}$ and where $\omega\neq 1$ is a (fixed) non-zero complex number. As will be clear from the next section these algebras have a natural $\mathbb{Z}_3$-grading.

Using these algebras we are going to create the necessary algebraic and analytic tools for the establishment of the counterpart of Hida's white noise space theory, including the construction of a topological algebra associated with the above mentioned algebra based on a decreasing family of Hilbert spaces which allows us to obtain the link between this algebra and the Fock space. We will finish by showing the application to stochastic processes.

\section{Finite Ternary Grassmann Algebras}
\label{sec-ternary}
Let us start by introducing ternary Grassmann algebras. Let $\fG= \{ \e_j, j \in \BN \}$\ be a countable set of linearly independent vectors over $\BC.$ We denote by  $V_d$ the complex linear space generated by the first $d$ of such vectors $\e_1, \cdots, \e_d$.

\begin{definition}\label{def2:001} We define the \textit{ternary Grassmann algebra} $\cG_{3, d}$ associated to $V_d$ as the free (non-trivial) algebra over $\BC$ containing a copy of $\Span\{ \e_1, \ldots, \e_d \}$ and of $\BC$, and satisfying to the following relations:
\begin{enumerate}[(i)]
\item there exists a complex $\omega \not=0$ such that
\begin{equation}\label{Eq:2.001}
\e_i\e_j=\omega \e_j\e_i, \qquad \mbox{\rm  for all }i<j;
\end{equation}
\item the basis elements satisfy
\begin{equation}\label{Eq:2.002}
\mathcal{T}(\e_i, \e_j, \e_k) =0, \qquad \mbox{\rm  for all } i, j, k =1, \ldots, d,
\end{equation}
\end{enumerate} where $\cT$ denotes the ternary form (based on the anti-commutator):
\begin{eqnarray}\label{Eq:2.003}
\mathcal{T}(\e_i, \e_j, \e_k) & = & \e_i \{ \e_j, \e_k\} + \e_j \{ \e_k, \e_i\} +\e_k \{ \e_i, \e_j\}   \nonumber \\
& = & \e_i \e_j \e_k +  \e_i \e_k \e_j + \e_j \e_i \e_k + \e_j \e_k \e_i + \e_k \e_i \e_j  + \e_k \e_j \e_i, \quad 1\leq i\leq j\leq k \leq d.
\end{eqnarray}
\end{definition}

\begin{remark} \label{Rem:2.001} By ``non-trivial'' we mean that the product of any two arbitrary basic elements of the algebra is either a new element of the algebra or it is zero. We will say that $\cG_{3, d}$ is generated by $\fG_d = \{ \e_1, \ldots, \e_d \} \subset \fG.$  Such algebras have already been studied in the literature. They are often referred to as {\it ternary Grassmann algebras}~\cite{T94, K2010}.
\end{remark}

An immediate consequence of the definition is the following lemma.

\begin{lemma}\label{lem:2.001} Under the conditions of Definition \ref{def2:001} it holds
\begin{enumerate}[(i)]
\item $\e_i^3 =0$ for all $i=1,\ldots, d;$
\item $\omega$ is a third root of the unit.
\end{enumerate}
\end{lemma}

\begin{proof} Proposition $(i)$ is immediate. For the second proposition we have for  $1\leq i < j  \leq d$
\begin{gather*}
0 = \cT(\e_i, \e_i, \e_j) = \e_i \e_i \e_j +  \e_i \e_j \e_i + \e_i \e_i \e_j + \e_i \e_j \e_i + \e_j \e_i \e_i  + \e_j \e_i \e_i \\
= 2 (\e_i^2 \e_j + \e_i \e_j \e_i + \e_j \e_i^2 )  = 2(1 + \omega +\omega^{2}) \e_i^2 \e_j.
\end{gather*} Under the assumption of non-triviality (i.e. $\e_j \e_i^2 \not=0 $) we obtain $\omega^2 + \omega + 1 =0$ and therefore,  $\omega$ is a root of third order the unity.
\end{proof}

\begin{remark} \label{Rem:2.002} In what follows we assume $\omega=e^{i2\pi/3}.$
\end{remark}

Due to Lemma \ref{lem:2.001} one observes that the relevant powers of the vector basis elements are given by $\e_j^m, m=0, 1, 2,$ since  $\e_j^m =0$ for all $m\geq 3$. Furthermore, we identify $\e_j^0 = 1,$ the identity of the field $\BC.$ Hence, we have
\begin{equation} \label{Eq:2:004}
\e_j^{m} = \left\{ \begin{array}{ll}
1, & m=0 \\
\e_j, & m=1 \\
\e_j^2, & m=2 \\
0, & m\geq 3
\end{array}        \right.
\end{equation} for all $j=1, \ldots, d.$ In consequence, a basis for the finite ternary Grassmann algebra $\cG_{3,d}$ is expressible in terms of appropriated ordered $d-$tuples of powers less than $3,$ that is,
\begin{equation} \label{Eq:2:005}
\e^{\bnu} := \e_1^{\nu_1} \cdots \e_d^{\nu_d}, \qquad \bnu = (\nu_1, \ldots, \nu_d) \in \{ 0, 1, 2 \}^d.
\end{equation} 
Under the previous convention we have $\e^{\boldsymbol{0}} =1$ with $\boldsymbol{0}=(0, \ldots, 0 ).$ We shall denote by $\mathcal I_d$ the set of all such $d-$tuples, that is,
\begin{equation} \label{Eq:2:005a}
\mathcal I_d = \{ 0, 1, 2 \}^d.
\end{equation}
Notice that we have $\e^{\bnu} = \e^{\bmu}$ if and only if $\nu_j = \mu_j$ for all $j=1, \ldots, d.$ 

Since there are $3^d$ such basis elements every  $z \in \cG_{3,d}$ can be written as
\begin{equation} \label{Eq:2:006}
\z = \sum_{|\bnu| =0}^{2d} z_{\bnu} \e^{\bnu}, \qquad z_{\bnu} \in \BC, ~|\bnu| = \nu_1 + \cdots +\nu_d.
\end{equation}
We observe that $\cG_{3,d}$ is a $Z_3$-graded algebra, that is, under multiplication the grades add up modulus $3$. This leads to the following blade decomposition of the generalised ternary Grassmann algebra:
\begin{equation} \label{Eq:2:007a}
\cG_{3,d} = [\cG_{3,d}]_0 \oplus [\cG_{3,d}]_1 \oplus \cdots \oplus [\cG_{3,d}]_{2d}
\end{equation}
where  each $k-$blade is
\begin{equation} \label{Eq:2:007b}
[\cG_{3,d}]_k = \left\{ \z \in \cG_{3,d} : \z = \sum_{|\bnu| =k} z_{\bnu} \e^{\bnu} \right\}, ~k=0, \ldots, 2d.
\end{equation} In what follows we denote by $[\z]_k := \sum_{|\bnu| =k} z_{\bnu} \e^{\bnu}$ the projection of $\z$ into the blade $[\cG_{3,d}]_k, k=0, \ldots, 2d.$

\begin{example} \label{Ex2:001} For dimension $d=2$ we have
$$\z = \underbrace{z_{00}}_{\in [\cG_{3,2}]_0} + \underbrace{z_{10}\e_1 + z_{01}\e_2}_{\in [\cG_{3,2}]_1} + \underbrace{z_{20}\e_1^2 + z_{11} \e_1 \e_2 + z_{02}\e_2^2}_{\in [\cG_{3,2}]_2} + \underbrace{z_{21}\e_1^2 \e_2 + z_{12}\e_1\e_2^2}_{\in [\cG_{3,2}]_3} +\underbrace{z_{22} \e_1^2\e_2^2}_{\in [\cG_{3,2}]_4},$$
where
$$[\z]_0 = z_{00}, \quad [\z]_1 = z_{10}\e_1 + z_{01}\e_2, \quad \quad [\z]_2 = z_{20}\e_1^2 + z_{11} \e_1 \e_2 + z_{02}\e_2^2,$$
$$ [\z]_3 =z_{21}\e_1^2 \e_2 + z_{12}\e_1\e_2^2, \quad \quad [\z]_4 = z_{22} \e_1^2\e_2^2.$$
\end{example}

The generalized finite ternary Grassmann algebra $\cG_{3,d}$ decomposes itself into the sum of three spaces
\begin{equation} \label{Eq:2.026}
\cG_{3,d} = \cG_{3,d}^{0} \oplus  \cG_{3,d}^{1} \oplus  \cG_{3,d}^{2},
\end{equation}
where each $\cG_{3,d}^{k}$
$$\cG_{3,d}^{k} = {\rm span}\,\{ \e^{\bnu} : |\bnu | = k \od  \}, \quad k=0,1,2.$$

We observe that these spaces obey the following multiplication rules:
\begin{gather*}
\cG_{3,d}^{0} \cdot \cG_{3,d}^{0} \subset \cG_{3,d}^{0},\qquad  \cG_{3,d}^{1} \cdot \cG_{3,d}^{2} \subset \cG_{3,d}^{0},\qquad  \cG_{3,d}^{2} \cdot \cG_{3,d}^{1} \subset \cG_{3,d}^{0}.
\end{gather*}
Hence, only $\cG_{3,d}^{0}$ is a subalgebra of $\cG_{3,d}$ while the spaces $\cG_{3,d}^{1}, \cG_{3,d}^{2}$ \textit{do not form an algebra}.

\subsection{Properties}

We now present results on products  in $\cG_{3,d}$. Due to (\ref{Eq:2.001}), we have
\begin{equation} \label{Eq:2:008b}
\e_i\e_j=\omega \e_j\e_i \quad \Leftrightarrow \quad \e_j\e_i = \omega^2 \e_i\e_j,
\end{equation}
for all $i<j.$ 

Hence, for each component we have
\begin{equation} \label{Eq:2:011}
\e_j^{\nu_j}  \e_j^{\mu_j} = \left\{ \begin{array}{ll}
\e_j^{\nu_j + \mu_j},& \mbox{if ~} 0 \leq \nu_j + \mu_j \leq 2 \\
0,& \mbox{otherwise.}
\end{array} \right., \quad j=1, \ldots, d.
\end{equation}

We now observe that for all $\bnu, \bmu \in \mathcal I_d$ we have
\begin{equation} \label{Eq:2:008a}
\e^{\bnu} \e^{\bmu} = 0,
\end{equation}
whenever $\nu_j + \mu_j \geq3$ for some $j.$ However, if $0 \leq \nu_j + \mu_j < 3$ for all $j$ then by (\ref{Eq:2:008b}) we obtain
\begin{equation} \label{Eq:2:008}
\e^{\bnu} \e^{\bmu} = \sigma(\bnu, \bmu)  ~\e^{\bnu + \bmu},
\end{equation}
where $\sigma(\bnu, \bmu) = \omega^{2 \sum_{s=1}^{d-1} \sum_{j=s+1}^d \nu_j \mu_s}$ corresponds to $\omega = \exp(\frac{2\pi i}{3})$ to the power of the number of permutations of the basis elements.
This leads to the following multiplication rule:
\begin{equation} \label{Eq:2:008c}
\e^{\bnu} \e^{\bmu} = \sigma(\bnu, \bmu) \e^{\bnu + \bmu}, \qquad \sigma(\bnu, \bmu):= \left\{ \begin{array}{ll}
0, & \mbox{\rm if } \bnu+\bmu \notin \mathcal I_d \\
\omega^{2 \sum_{s=1}^{d-1} \sum_{j=s+1}^d \nu_j \mu_s}, & \mbox{\rm otherwise}
\end{array}        \right. .
\end{equation}

\begin{lemma} \label{lem:2.003} For every $\z \in [\cG_{3,d}]_k, \w \in [\cG_{3,d}]_s,$ $(0\leq k, s\leq 2d),$ we have that $\z \w =0$ if $k+s > 2d$ and
\begin{equation} \label{Eq:2:010}
\z \w \in \{ 0 \} \oplus [\cG_{3,d}]_{k+s} ,\quad  {\rm if ~} 0\leq k+s \leq 2d.
\end{equation}
\end{lemma}
\begin{proof} Hence, we obtain
\begin{gather*}
\z \w = (\sum_{|\bnu| =k} z_{\bnu} \e^{\bnu}) (\sum_{|\bmu| =s} w_{\bmu} \e^{\bmu})= \sum_{{|\bnu| =k, |\bmu| =s}\atop { \bnu+\bmu \in \mathcal I_d}} z_{\bnu} w_{\bmu} \sigma(\bnu, \bmu) \e^{\bnu+\bmu}\\
= \sum_{{|\bnu| =k, |\bmu| =s}\atop { \bnu+\bmu \in \mathcal I_d}} \sigma(\bnu, \bmu) ~z_{\nu_1, \ldots, \nu_d} w_{\mu_1, \ldots, \mu_d} \e_1^{\nu_1+\mu_1} \cdots \e_d^{\nu_d+\mu_d}.
\end{gather*} The result follows trivially from (\ref{Eq:2:008c}). 
 \end{proof}

\begin{lemma} \label{lem:2.004} For every vector $\z = z_1 \e_1 + \cdots + z_d \e_d \in [\cG_{3,d}]_1$ it holds
\begin{equation} \label{Eq:2:012}
\z^3 = 0.
\end{equation}
\end{lemma}
\begin{proof} By direct computation we get
\begin{gather*}
\z^3 = \sum_{i,j,k} z_iz_jz_k \cT(\e_i, \e_j, \e_k)=0.
\end{gather*}
since $\omega$ is the 3rd root of the unit and $\e_i^3=0.$
 \end{proof}

 We now present some decomposition results for the finite ternary Grassmann algebra. The results will be expressed in terms of the basis element $\e_d$ but are easily extendable to any basis element $\e_j$ with appropriate modifications.

\begin{lemma} \label{lem:2.004a} Every element $\z \in \cG_{3,d}$ admits the following decomposition
\begin{equation} \label{Eq:2:012a}
\z  = A + B \e_d + C \e_d^2, \quad A, B, C \in \cG_{3, d-1}.
\end{equation}
\end{lemma}

The result is straightforward and its proof will be omitted.
\begin{lemma} \label{lem:2.004b} For every $A \in \cG_{3,d-1},$ there exists $A' \in \cG_{3,d-1}$ such that
\begin{equation} \label{Eq:2:012b}
A  \e_d = \e_d A'.
\end{equation}
\end{lemma}

\begin{proof} Recall that $\e_j \e_d = \omega \e_d \e_j, ~j=1, \cdots, d-1.$ Hence, we have for $A = \sum_{\bnu \in \cI_{d-1}} a_{\bnu} {\e}^{\bnu} \in \cG_{3,d-1}$
\begin{gather*}
A\e_d =  \left(\sum_{ \bnu \in \cI_{d-1}} a_{\bnu} \e^{\bnu} \right) \e_d = \sum_{\bnu \in \cI_{d-1}} a_{\bnu} \e_1^{\nu_1} \cdots \e_{d-2}^{\nu_{d-2}} (\e_{d-1}^{\nu_{d-1}} \e_d) \\
= \sum_{\bnu \in \cI_{d-1}} a_{\bnu}  \e_1^{\nu_1} \cdots \e_{d-2}^{\nu_{d-2}} \left( \omega^{\nu_{d-1}} \e_d \e_{d-1}^{\nu_{d-1}}\right) = \cdots = \sum_{\bnu \in \cI_{d-1}} a_{\bnu} \omega^{\sum_{j=1}^d \nu_j} \e_d \e_1^{\nu_1} \cdots \e_{d-2}^{\nu_{d-2}} \e_{d-1}^{\nu_{d-1}} \\
= \e_d \left( \sum_{\bnu \in \cI_{d-1}} a_{ \bnu} \omega^{\sum_{j=1}^d \nu_j}  \e_1^{\nu_1} \cdots \e_{d-2}^{\nu_{d-2}} \e_{d-1}^{\nu_{d-1}} \right) := \e_d A'.
\end{gather*}

\end{proof}

\begin{corollary} \label{cor:2.004a} If we have $A=0$ in decomposition (\ref{Eq:2:012a}) then $\z^3=0.$
\end{corollary}

This is an obvious consequence of the two previous lemmas, as $\z = (B+C\e_d)\e_d.$

\begin{lemma} For every $1\leq n \leq d$ there exists an element $I_n = \e^2_{d-n+1} \cdots \e^2_d$ such that it satisfies:
\begin{enumerate}[i)]
\item $I_n$ is nilpotent, that is $I_n^2=0;$
\item $I_n$ acts as a projector of $\cG_{3,d}$ onto the subalgebra $\cG_{3,d-n},$ that is to say, there exists a projector $P_n : \cG_{3,d} \rightarrow \cG_{3,d-n}$ given by
$$\z= \sum_{\bmu \in \cI_d} z_{\bmu} \e^{\bmu} \in \cG_{3,d} \mapsto P_n(\z): = \sum_{ \bnu \in \cI_{d-n}} z_{\bnu} \e^{\bnu} \in \cG_{3,d-n},$$ where $ \bnu = (\nu_1, \cdots, \nu_{d-n});$
\item In particular, $P_d$ acts as a projector of $\cG_{3,d}$ onto $\BC$, given by $\z= \sum_{\bmu \in \cI_d} z_{\bmu} \e^{\bmu} \in \cG_{3,d} \mapsto P_d (\z) := [\z]_0 = z_{\boldsymbol{0}} := z_{(0, \cdots, 0)} .$
\end{enumerate}
\end{lemma}

\begin{proof} The fact that $I_n$ is nilpotent is straightforward. For the projection part, we observe that by (\ref{Eq:2:011}) the product of an arbitrary $\z$ by $I_n$ kills off all terms containing $\e_{d-n+1}, \cdots, \e_d,$ that is, given
$$\z = \sum_{\bmu  \in \cI_d} z_{\bmu} \e^{\bmu} = \sum_{\bmu  \in \cI_{d-n}} z_{\bmu} \e^{\bmu} + \sum_{\bmu  \in \cI_d \setminus \cI_{d-n}} z_{\bmu} \e^{\bmu},$$ we have
$$\z I_n = (\sum_{\bmu  \in \cI_d} z_{\bmu} \e^{\bmu}) I_n = \sum_{\bmu  \in \cI_{d-n}} z_{\bmu} \e^{\bmu} I_n, \qquad  \bmu = (\mu_1, \cdots, \mu_{d-n}).$$
Hence, we identify the projection of $\z$ into $\cG_{3,d-n}$ with $P_n(\z) := \sum_{\bmu \in \cI_{d-n}} z_{\bmu} \e^{\bmu}.$ The third proposition is now immediate.
 \end{proof}

\begin{remark} \label{Rem:2.004} Based on the last proposition, the non-scalar part of an element $\z \in \cG_{3,d}$ is obtained by
$$ (1-P_d) \z = \sum_{{\bmu \in \cI_d} \atop {\bmu \not= \boldsymbol{0}}} z_{\bmu} \e^{ \bmu}.$$ Hence,
$$ \z =P_d \z + (1-P_d) \z.$$
Henceforward we shall use the notations $z_{\boldsymbol{0}} := P_d \z$ for its scalar part (also, \textit{body}  of $\z$) and $\z_r := (1-P_d) \z = \sum_{{\bmu \in \cI_d} \atop {\bmu \not= \boldsymbol{0}}} z_{\bmu} \e^{\bmu}$ for its remainder (also, \textit{soul}  of $\z$).
\end{remark}

\begin{lemma} \label{lem:2.005} An arbitrary element $\z \in \cG_{3,d}$ is invertible if and only if its scalar part $[\z]_0 = z_{\boldsymbol{0}}$ is non-zero.
\end{lemma}

\begin{proof} We begin our proof by showing that if the scalar part of an element $\z \in \cG_{3,d}$ is zero then this element cannot be invertible. Indeed, if $[\z]_0 =0$ then $\z = \sum_{{\nu \in \cI_d} \atop {\nu \not= \mathbf{0}}} z_\nu \e^\nu$ and by Lemma \ref{lem:2.003} we get $[\z \w]_0 = 0$ for all $\w \in \cG_{3,d}.$ Hence, $\z$ is not invertible.

Now we consider $[\z]_0 \not=0.$ First, we observe that there exists $m \in \BN$ such that $\z_r^m =0.$ Take  $m(\z) := \min \{m \in \BN : \z_r^m =0 \}.$ Whenever $z_{\mathbf{0}} \not=0$ we have
\begin{gather*}
\z^n = z_{\mathbf{0}}^n \left(1+\frac{1}{z_{\mathbf{0}}}\soul{\z} \right)^n, \quad n\in \BN.
\end{gather*}

Hence, for $n= m(\soul{\z}) -1$ we obtain
\begin{gather*}
\left(1+ \frac{1}{z_{\mathbf{0}}}\soul{\z} \right) \left(1- \frac{1}{z_{\mathbf{0}}}\soul{\z} + \frac{1}{z_{\mathbf{0}}^2}\soul{\z}^2 + \cdots +  (-1)^{m(\z)-1}  \frac{1}{z_{\mathbf{0}}^{m(\z)-1}}\soul{\z}^{m(\z)-1} \right) =1,
\end{gather*} so that
\begin{gather*}
\z^{-1} = \frac{1}{z_{\mathbf{0}}}- \frac{1}{z_{\mathbf{0}}^2}\soul{\z} + \frac{1}{z_{\mathbf{0}}^3}\soul{\z}^2 + \cdots +  (-1)^{m(\z)-1}  \frac{1}{z_{\mathbf{0}}^{m(\z)}}\soul{\z}^{m(\z)-1},
\end{gather*}  is the right inverse of $\z.$ The same construction holds for the left inverse which proves the unicity of $\z^{-1}$.
 \end{proof}

\subsection{A conjugation  in $\cG_{3,d}$}

We present a morphism acting on the finite generalised ternary algebra $\cG_{3,d}.$

\begin{definition}[Pseudo-conjugation]
The {\it pseudo-conjugation} in the ternary Grassmann algebra $\cG_{3,d}$ is defined as the morphism $\ov{\cdot} : \cG_{3,d} \rightarrow \cG_{3,d},$ with $\z \mapsto \ov{\z} =\sum_{\bnu} {\overline{z}}_{\bnu} ~\ov{\e^{\bnu}},$
where $\overline{z}_{\bnu}$ denotes the standard complex conjugation, while its action on the basis elements $\e^{\bnu}, \bnu \in \cI_d,$ is given by,
\begin{equation}
\ov{1} = 1, \quad \ov{\e}_j = \e_j^{2}, \quad j=1, \cdots, d,
\end{equation} and satisfying to
\begin{equation}
\ov{\mathbf{ab} + \mathbf{c}} = \ov{\mathbf{b}} ~\ov{\mathbf{a}} + \ov{\mathbf{c}}, \quad for ~all ~ \mathbf{a}, \mathbf{b}, \mathbf{c} \in \cG_{3,d}.
\end{equation}
\end{definition}

 \begin{remark} \label{Rem:2.006} As a consequence, we get $\ov{\e_j^2} = \ov{\e}_j \ov{\e}_j = \e_j^2 \e_j^2= 0.$ Hence, $\ov{\e^{\bnu}} =0$ if and only if $\bnu \not\in \{0, 1\}^d.$ Furthermore this morphism is not onto and it is not an involution since $\ov{(\ov \e_j)}=0$.
\end{remark}

\begin{lemma}   For all $\z \in \cG_{3,d}$ it holds:
\begin{enumerate}[i)]
\item $\ov{(\ov \z)} = z_{\mathbf{0}};$
\item $[\z \ov{\z}]_0 = [\ov{\z} \z]_0 =|z_{\mathbf{0}}|^2,$
\end{enumerate} where we recall, $z_{\mathbf{0}}$ denotes the scalar part of $\z.$
\end{lemma}

\begin{proof} The first proposition is obvious since the action of the pseudo-conjugation on the basis elements is given by
$$\ov{\e^{\bnu}}  =  \e_d^{2\nu_d} \cdots \e_2^{2\nu_2}\e_1^{2\nu_1},\qquad \mbox{\rm whenever } \bnu \in \{0, 1\}^d,$$ and zero otherwise.
Furthermore, as $\e_j^2 \e_i^2 = \omega^2 \e_i^2 \e_j^2, ~i<j,$ we obtain
$$\ov{\e^{\bnu}} = \sigma(\bnu, \bnu) \e^{2\bnu} := \omega^{2(\sum_{j=1}^{d-1} \sum_{s=j+1}^d \nu_j \nu_s)} \e^{2\bnu}, \quad \bnu \in \{0, 1\}^d.$$
Hence,
\begin{gather*}
\ov{(\ov \z)}  =  \ov{\sum_{\bnu \in \{0, 1\}^d} \ov z_{\bnu} \e_d^{2\nu_d} \cdots \e_2^{2\nu_2}\e_1^{2\nu_1}}  = \sum_{\bnu \in \{0, 1\}^d} z_{\bnu} \ov{\e_1^{2\nu_1}}~ \ov{\e_2^{2\nu_2}} \cdots \ov{\e_d^{2\nu_d}} = \sum_{\bnu \in \{0, 1\}^d} z_{\bnu} {\e_1^{4\nu_1}}~ {\e_2^{4\nu_2}} \cdots {\e_d^{4\nu_d}} = z_{\mathbf 0}.
\end{gather*}

For the second proposition we have
\begin{gather*}
[\z \ov{\z}]_0 = \left[ \left( \sum_{\bnu \in \cI_d} z_{\bnu} \e^{\bnu} \right) \left( \ov{\sum_{\bmu\in \cI_d} z_{\bmu} \e^{\bmu}} \right) \right]_0 \\
= \left[ \sum_{{\bnu \in \cI_d} \atop {\bmu \in \{ 0,1\}^d} } z_{\bnu} \ov{z}_{\bmu} \e_1^{{\nu}_1} \e_2^{{\nu}_2} \cdots  \e_d^{{\nu}_d} \e_d^{2\mu_d} \cdots \e_2^{2\mu_2}\e_1^{2\mu_1} \right]_{0} = z_{\boldsymbol{0}}  \ov z_{\boldsymbol{0}} = |z_{\boldsymbol{0}} |^2.
\end{gather*}
The same holds for $[\ov{\z} \z]_0,$ which completes our proof.
\end{proof}

\section{Completions of Grassmann Algebras}
\label{sec-completions}

We now consider the ternary Grassmann algebra generated by taking the formal limit $d\rightarrow \infty$ of $\cG_{3,d}.$ We denote the correspondent algebra by $\cG_3,$ associated to the countable set $\fG= \{ \e_j, j \in \BN \}.$ Similar to the case of the infinite dimensional Grassmann algebra $\Lambda_\infty$ the resulting ternary Grassmann algebra $\cG_3$  is an associative but not commutative algebra over $\BC.$

We denote its elements $\z = \sum_{\bnu \in \cI} z_{\bnu} \e^{\bnu} \in \cG_3$ as \textit{ternary supernumbers} where  $\cI = \{ 0,1, 2\}^{\BN}$ denotes the set of indexes, and we endow the ternary Grassmann algebra $\cG_3$ with a $p$-norm. Remark that, since the pseudo-conjugation is not an isomorphism it does not induce a norm. Hence, we will use the $\ell^p$-norm where $\z = \sum_{\bnu \in \cI} z_{\bnu} \e^{\bnu}$ is to be identified with $(z_{\bnu})_{\bnu \in \cI} \in \ell^p(\BC).$

The conjugation and product in $\cG_3$ (an infinite dimensional algebra) are well defined provide that  $\bnu$ satisfy $\# \bnu < \infty,$ where $\# \bnu$ denotes the number of non-zero entries in the sequence  $\bnu$ (equal to the number of $\e_j's$ present in the basis element $\e^{\bnu}$). For example, for $\bnu = (1, 0, 2, 1, 0,0, \ldots)$ we get $\# \bnu = 3$ corresponding to  $\e^{\bnu} = \e_1 \e_3^2\e_4.$

Hence,  we have for the conjugation
\begin{gather} \label{Eq:3:InftyConjugation}
\mathbf{w} = \sum_{\bmu \in \cI} w_{\bmu} \e^{\bmu} \quad \mapsto \quad  \overline{\w} = \sum_{\bmu \in \{ 0, 1\}^{\BN}} \sigma(\bmu, \bmu) ~\overline w_{\bmu} \e^{2\bmu},
\end{gather}
and for the product between $\z = \sum_{\bnu \in \cI} z_{\bnu} \e^{\bnu}, \mathbf{w} = \sum_{\bmu \in \cI} w_{\bmu} \e^{\bmu}\in \cG_3$ we get
\begin{gather} \label{Eq:3:InfinitProduct}
\z \w = \sum_{\bnu,\bmu \in \cI} \sigma(\bnu, \bmu) ~z_{\bnu} w_{\bmu} \e^{\bnu+\bmu},
\end{gather} under the restriction $\# \bnu, \# \bmu$ are finite, and where $\sigma(\bnu, \bmu)$ is defined as in (\ref{Eq:2:008c}).  We stress that only under particular conditions we have $\sigma(\bnu, \bmu) = \overline{\sigma(\bmu, \bnu)}.$

Henceforth we assume all sequences to have a finite number of non-zero entries.

\begin{definition}\label{Def:4.1}
Let $p\in \BN.$ We define the $p$-norm of $\z\in\cG_{3}$ is defined as
\begin{equation}
\| \z \|_p = \left(\sum_{\bnu \in \cI} |z_{\bnu} |^p \right)^{1/p},
\label{pnorm}
\end{equation}
where $|\cdot|$ is the usual modulus of a complex number.
\end{definition}

We remark that Definition \ref{Def:4.1} holds also for any real $p \geq 1.$

Restricted to the finite dimensional sub-algebra $\cG_{3,d}$ the $p$-norm satisfy the following properties:

\begin{theorem} \label{Th:3.1}
For all $\z, \w\in\cG_{3,d}$ it holds
\begin{enumerate}[i)]
\item \begin{equation} \label{1norm-ineq}
\Vert \z\w \Vert_1 \leq \Vert \z \Vert_1 \Vert \w \Vert_1,
\end{equation}
\item and for $p=2, 3, \ldots$
\begin{equation} \label{norm-ineq}
\Vert \z \w \Vert_p^p \leq \Vert \z \Vert_1^p \Vert \w \Vert_{2^{p-1}} \prod_{k=1}^{p-1} \Vert \w\Vert_{2^k}, \quad \Vert \z \w \Vert_p^p \leq \Vert \w \Vert_1^p \Vert \z \Vert_{2^{p-1}} \prod_{k=1}^{p-1} \Vert \z\Vert_{2^k}.
\end{equation}
\end{enumerate}
\end{theorem}

The proof of \eqref{1norm-ineq} is straightforward. Moreover, the proof of \eqref{norm-ineq} follows in the same as the one presented in \cite{alpay2018distribution}, where the Cauchy-Schwarz inequality is used repeated times and having in mind that  $|\omega|=1$.

In order to study analytic properties of stochastic processes taking values in this algebra one needs to consider its completion with respect to the $\ell^2-$norm. In the next section we study the completion of $\cG_{3}$ with respect to the $p$-norm, which we denote by $\comp{p}.$ This closure is widely studied in the literature in the classical case of Grassmann algebras (see, e.g., \cite{alpay2018positivity, rogers1980global, zbMATH00861741}). Also, remark that by \eqref{1norm-ineq} we have that $\comp{1}$ has a Banach algebra structure. The study of completion of the $p-$norm has the purpose of establish a ternary Fock space based on the $\ell^2$-inner product between two supernumbers $\z, \w\in \cG_3$ given by
\begin{equation} \label{Eq: 3_l2innerproduct}
\langle \z, \w\rangle = \sum_{\bnu \in \cI} z_{\bnu} \overline{w_{\bnu}}.
\end{equation}

\subsection{Topological Algebra Associated with $\cG_3$}
\setcounter{equation}{0}
\label{sec-topological}

In order to establish an analysis and stochastic process theory in the framework of ternary Grassmann algebras we need to establish an equivalent to the classic Gel'fand triple $(\cS,\mathbf L^2(\mathbb R,dx),\cS^\prime)$ where $\cS$ is the space of test functions and $\cS^\prime$ denotes its dual (see for instance \cite{MR0209834}). The commutative setting was first introduced by Kondratiev and adapted to the framework of Hida's white noise space theory and commutative Fock space (see \cite{MR1408433,MR1387829}). More, the commutative case is associated with bosons and applied to study solutions of  stochastic differential equations and to model stochastic processes and their derivatives. The construction of the noncommutative counterpart of this theory in the classic setting (see e.g. \cite{MR3231624,MR3038506}) was motivated by fermionic formulation. In our case of topological hypersymmetry we need to establish the corresponding noncommutative counterpart in terms of our algebra and to construct a Gel'fand triple in this ternary Grassmann setting.

Gel'fand triples allow to define other products (on itself not necessarily laws of composition) different from the usual inner product in the Hilbert space.  One such example is the Wick product in the white noise space which is not a law of composition. By embedding the white noise space into an analogous space of stochastic distributions the Wick product becomes a law of composition by strict inclusion. Another important reason such a construction of the space of stochastic distributions is the fact that such spaces are necessary for a future study of their derivatives (see \cite{MR3231624}).

Hence, we now recall a few facts from the classical case and from the theory of perfect spaces and strong algebras. For more details on these spaces we refer the reader to \cite{GS2_english, MR0435834}.\smallskip

Starting from a decreasing family of Hilbert spaces $(\mathcal H_p,\|\cdot\|_{\mathcal{H}_p})_{p\in\mathbb Z}$, with increasing norms,
$$ \ldots \subseteq \mathcal{H}_2 \subseteq \mathcal{H}_1\subseteq \mathcal{H}_0 \subseteq \mathcal{H}_{-1}\subseteq \mathcal{H}_{-2} \ldots$$ it is known that the intersection $\mathcal F=\cap_{p=0}^\infty \mathcal H_p$ is a Fr\'echet space. If, furthermore, $\mathcal F$ is perfect then compactness is equivalent to being compact and bounded. In particular, this is ensured when for each $p$ there exists $q>p$ such that the injection map from ${\mathcal H_q}$ into ${\mathcal H_p}$ is compact. 
In the usual way we identify ${\mathcal H}_p^\prime$ with ${\mathcal H}_{-p}$. Then $\mathcal F$ together with its the dual $\mathcal F^\prime :=\cup_{p=0}^\infty{\mathcal H}_{-p}$ and $\mathcal H_0$ forms a Gel'fand triple $(\mathcal F, \Gamma(\mathcal H_0), \mathcal F^\prime)$. Indeed, the dual $\mathcal F^\prime$ endowed with the strong topology defined in terms of the bounded sets of $\mathcal F$ is then locally convex. Furthermore, the strong topology coincides with the inductive limit topology. See \cite[Section 3]{MR3029153} for a discussion.  Thus, the space of distributions $\mathcal F^\prime$ is the dual of a Fr\'echet nuclear space.

We recall here two statements about compactness and convergence of sequences in $\mathcal F^\prime$.

\begin{proposition}
\label{compact}\cite{GS2_english}
A set is (weakly or strongly) compact in $\mathcal F^\prime$ if and only if it is compact in one of the spaces ${\mathcal H}_{-p}$ in the corresponding norm.
\end{proposition}

\begin{proposition}\cite{GS2_english}
Assume $\mathcal F^\prime$ perfect. Then, weak and strong convergence of sequences are equivalent, and a sequence converges (weakly or strongly) if and only if it converges in one of the spaces ${\mathcal H}_{-p}$ in the corresponding norm.
\label{convergence}
\end{proposition}

Now, we are going to show that $\mathcal F$ can be made a topological algebra denoted $\mathfrak{S}_{1}$ where the product satisfies the so-called V\r{a}ge inequality. This ensures that we can consider $\mathcal F^\prime$ as an inductive can be done as limit of Hilbert spaces.  This is used in the proof of Theorem \ref{thm55}.

A topological algebra is assumed to be separately continuous in each variable. It is not immediate, but true that a strong algebra is jointly continuous in the two variables (see \cite[IV.26, Theorem 2]{Bourbaki81EVT} and also the discussion in \cite[pp. 215-216]{MR3404695}).

In our case, we define
\begin{equation}
\label{hp123}
\mathcal{H}_{p}(\mathbf{c}) = \left\{f = \sum_{\bnu \in \cI} f_{\bnu} \e^{\bnu} \in \comp{2} \suchthat \sum_{\bnu \in \cI} | f_{\bnu} |^2 c_{\bnu}^{2p}<\infty \right\},
\end{equation}
with $p\in\mathbb{Z}$. The coefficients give rise a sequence $\mathbf{c} = (c_{\bnu})_{\bnu \in \cI}$ of positive real numbers such that
\begin{equation}
c_{\bnu} c_{\bmu} \leq c_{\bgamma}, \quad \mbox{for  all }\bnu, \bmu \in \cI \mbox{ such that } \bnu+\bmu=\bgamma \in \cI,
\label{cond1}
\end{equation}
and where
\begin{equation}
\sum_{\bnu \in \cI} c_{\bnu}^{-2d} < \infty,\quad \mbox{for } d=1, 2, 3, \ldots
\label{cond2}
\end{equation}
By construction we have
$$\mathcal{H}_{-q}(\mathbf{c}) \subseteq \mathcal{H}_{-p}(\mathbf{c}),$$
if $p\geq q$.

From here on we abbreviate $\mathcal{H}_{-p}(\mathbf{c})$ by $\mathcal{H}_{-p}$.

\begin{definition}
The norm $\| \cdot \|_{\mathcal{H}_{-p}}$ in $\mathcal{H}_{-p}$ is defined as
$$
\| f \|_{\mathcal{H}_{-p}} := \sum_{\bnu \in \cI} | f_{\bnu} |^2 c_{\bnu}^{-2p}.
$$
\end{definition}

\begin{proposition}
If $c_{\bnu} c_{\bmu} = c_{\bnu + \bmu}$, then $c_{\mathbf{0}}=1$.
\label{c0}
\end{proposition}

\begin{proof} If so, then
$$c_{\mathbf{0}} c_{\bmu} = c_{\mathbf{0} + \bmu} = c_{\bmu} $$ so that $c_{\mathbf{0}}=1.$
\end{proof}

\begin{proposition} Let $\mathbf{c} = (c_{\bnu})_{\bnu \in \cI}$ be such that $c_{\mathbf{0}}=1$ and $c_{\bnu} >1,$ for all $\bnu \not= \mathbf{0}.$ Then,
$$\lim_{p\rightarrow\infty}\left\Vert f\right\Vert_{\mathcal{H}_{-p}} = |f_{\mathbf{0}}|^2, \quad \mbox{for all }f\in\mathcal{H}_{-p}.$$
\label{limit-norm}
\end{proposition}

\begin{proof}
Since $\lim_{p\rightarrow\infty} c_{\bnu}^{-2p}=0$ for every $\bnu \not= \mathbf{0}$,
\begin{eqnarray*}
\lim_{p\rightarrow\infty}\| f \|_{\mathcal{H}_{-p}} & = & \lim_{p\rightarrow\infty}\sum_{\bnu \in \cI} |f_{\bnu} |^2 c_{\bnu}^{-2p} \\
    & = & \sum_{\bnu \in \cI} | f_{\bnu}|^2 \left( \lim_{p \rightarrow\infty} c_{\bnu}^{-2p} \right) \\
    & = & | f_{\mathbf{0}}|^2.
\end{eqnarray*}
\end{proof}

\begin{definition}
We consider the space
\begin{equation} \label{S_1}
\mathfrak{S}_1 = \cap_{p \geq 0}
\mathcal{H}_p
\end{equation}
and its topological dual
\begin{equation} \label{S_1dual}
\mathfrak{S}_{-1}=\cup_{p \geq 0}
 \mathcal{H}_{-p},
\end{equation}
which can be considered as analogues of the spaces $\cS$ and $\cS^\prime$, respectively, in our setting.
\end{definition}

The next theorem introduces a V\r{a}ge-like inequality \cite{vage96} which permits the analysis of stochastic processes to be done locally in a Hilbert space.

\begin{theorem}
If $f\in\mathcal{H}_{-q}$ and $g\in\mathcal{H}_{-p}$, with $p> q$, then
\begin{equation}
\label{vage_ineq}
\left\| fg \right\|_{\mathcal{H}_{-p}} \leq C_{p-q} \left\Vert f\right\Vert_{\mathcal{H}_{-q}} \left\Vert g\right\Vert_{\mathcal{H}_{-p}}, \quad \left\| gf \right\|_{\mathcal{H}_{-p}} \leq C_{p-q} \left\Vert f\right\Vert_{\mathcal{H}_{-q}} \left\Vert g\right\Vert_{\mathcal{H}_{-p}},
\end{equation}
with $C_{p-q}>0$ being a constant.
\label{ineq}
\end{theorem}

\begin{proof} Suppose $f\in\mathcal{H}_{-q}$ and $g\in\mathcal{H}_{-p}$. Applying Cauchy-Schwarz inequality we get
\begin{eqnarray*}
\left\| fg \right\|_{\mathcal{H}_{-p}}^2 & = & \sum_{\bgamma\in\cI} |(fg)_{\bgamma}|^2 c_{\bgamma}^{-2p} \\
     & = & \sum_{\bgamma\in\cI} \left| \sum_{{\bnu, \bmu \in \cI} \atop {\bnu+\bmu = \bgamma}}\sigma(\bnu, \bmu)f_{\bnu} g_{\bmu} \right|^2 c_{\bgamma}^{-2p} \\
     & \leq & \sum_{\bgamma\in\cI} \left( \sum_{{\bnu, \bmu, \bnu', \bmu' \in \cI} \atop {\bnu+\bmu = \bnu'+\bmu' = \bgamma}} |f_{\bnu} |  |g_{\bmu} | | f_{\bnu'} |  |g_{\bmu'} | \right)  c_{\bgamma}^{-2p} \\
     & \leq & \sum_{\bgamma\in\cI} \left( \sum_{{\bnu, \bmu, \bnu', \bmu' \in \cI} \atop {\bnu+\bmu = \bnu'+\bmu' = \bgamma}} | f_{\bnu} | c_{\bnu}^{-p}  |g_{\bmu} |c_{\bmu}^{-p}  | f_{\bnu'} | c_{\bnu'}^{-p} |g_{\bmu'} |  c_{\bmu'}^{-p} \right)  \\
      & \leq & \sum_{\bnu, \bnu' \in\cI} | f_{\bnu} | c_{\bnu}^{-p}  | f_{\bnu'} | c_{\bnu'}^{-p} \left( \sum_{{\bgamma\in\cI : \exists \bmu, \bmu' \in \cI} \atop {\bnu+\bmu = \bnu'+\bmu' = \bgamma}}   |g_{\bmu} |c_{\bmu}^{-p}   |g_{\bmu'} |  c_{\bmu'}^{-p} \right)  \\
          & \leq & \sum_{\bnu, \bnu' \in\cI} | f_{\bnu} | c_{\bnu}^{-p}  | f_{\bnu'} | c_{\bnu'}^{-p} \left( \sum_{{\bgamma\in\cI : \exists \bmu \in \cI} \atop {\bnu+\bmu = \bgamma}}   |g_{\bmu} |^2 c_{\bmu}^{-2p}   \right)^{\frac{1}{2}}  \left( \sum_{{\bgamma\in\cI : \exists \bmu' \in \cI} \atop {\bnu'+\bmu' = \bgamma}}   |g_{\bmu'} |^2  c_{\bmu'}^{-2p} \right)^{\frac{1}{2}}  \\
          & \leq & \sum_{\bnu, \bnu' \in\cI} | f_{\bnu} | c_{\bnu}^{-p}  | f_{\bnu'} | c_{\bnu'}^{-p} \left( \sum_{\bmu\in\cI }   |g_{\bmu} |^2c_{\bmu}^{-2p}   \right)^{\frac{1}{2}}  \left(  \sum_{\bmu' \in\cI }    |g_{\bmu'} |^2  c_{\bmu'}^{-2p} \right)^{\frac{1}{2}}  \\
          & \leq & \left( \sum_{\bnu  \in\cI} | f_{\bnu} | c_{\bnu}^{-p}   \right)^2 \| g \|^2_{\cH_{-p}}  \\
         & \leq & \left( \sum_{\bnu  \in\cI} | f_{\bnu} | c_{\bnu}^{-q}c_{\bnu}^{-(p-q)}   \right)^2 \| g \|^2_{\cH_{-p}}  \\
    & \leq & \left( \sum_{\bnu  \in\cI} c_{\bnu}^{-2(p-q)}   \right)  \| f \|_{\cH_{-q}}^2  \| g \|^2_{\cH_{-p}}.
\end{eqnarray*}

It remains to prove that there exists a sequence $\mathbf{c} = (c_{\bnu})$  such that $\sum_{\bnu \in \cI} c_{\bnu}^{-2d}<\infty$ for all $d=1,2, \ldots$. We assume this sequence to be given by
$$c_{\bnu} = e^{ \sum_k \varphi(3^{k-1} \nu_k)} = e^{\varphi(\nu_1) + \varphi(3 \nu_2) + \varphi(3^2 \nu_3) + \cdots },$$
where $ \nu_k \in \{ 0,1, 2\}.$ We now look into the properties of such a function $\varphi.$

Bering in mind that $c_{\mathbf{0}}=1$ and $c_{\bnu} c_{\bmu} = c_{\bnu+\bmu}$ for all $\bnu, \bmu \in \cI$ such that $\bnu+\bmu \in \cI,$ we obtain
\begin{enumerate}[i)]
\item $c_{\mathbf{0}} = e^{\sum_k \varphi(0)},$ leading to $\varphi(0)=0;$
\item since $c_{\bnu} c_{\bmu} = c_{\bnu+\bmu}$ if $\bnu+\bmu \in \cI,$ we have that for a given position $k \in \BN,$  $\nu_k +\mu_k \not=3, 4;$
\item $c_{\bnu} c_{\bmu} = c_{\bnu+\bmu}$ implies
$$e^{ \sum_k \varphi(3^{k-1} \nu_k)} e^{ \sum_j \varphi(3^{j-1} \mu_j)}  = e^{\sum_k \varphi \big(3^{k-1} (\nu_k+\mu_k)\big)};$$
\item furthermore, for integers $k > j$ and $\nu_k, \mu_j \in \{ 1, 2 \}$ we have $3^{k-1} \nu_k > 3^{j-1} \mu_j$ so that $\varphi$ should be an increasing function, satisfying to $\varphi(a) + \varphi(b) = \varphi(a+b),$ for $a, b >0.$
\end{enumerate}

For example, consider $\varphi(x) = x, ~x>0.$ Hence,
\begin{eqnarray*}
\sum_{\bnu  \in\cI} c_{\bnu}^{-2d} & = & 1 +  \sum_{{\bnu  \in\cI} \atop {|\bnu \not= \mathbf{0}}} c_{\bnu}^{-2d} \\
& = & 1 +  \sum_{{\bnu  \in\cI} \atop {|\bnu| \not= \mathbf{0}}} e^{ -2d \sum_k 3^{k-1} \nu_k}
\end{eqnarray*} Now, we split this sum in terms of the number $\# \bnu$ of non-zero entries in the sequence  $\bnu.$ Then
\begin{eqnarray*}
\sum_{\bnu  \in\cI} c_{\bnu}^{-2d} & = & 1 + \sum_{m=1}^{\infty} \sum_{{\bnu  \in\cI} \atop {\# \bnu =m}} e^{-2d \sum_k  3^{k-1} \nu_k}.
\end{eqnarray*} We observe that, for $m=1$ we have
\begin{eqnarray*}
 \sum_{{\bnu  \in\cI} \atop {\# \bnu =1}} e^{ -2d \sum_k  3^{k-1} \nu_k} & =&  \sum_{k=1}^{\infty} e^{ -2d 3^{k-1} \nu_k} \leq \sum_{k=1}^{\infty} e^{ -2 d 3^{k-1}} \qquad (\mbox{recall: } \nu_k = 1, 2)\\
 & \leq &  \sum_{k=1}^{\infty} e^{ -2d k} = e^{-2d} \frac{1}{1-e^{-2d}}< \infty.
\end{eqnarray*} Furthermore, remark that $0< e^{-2d} < 1, ~d=1,2, \ldots$ so that $ \frac{1}{1-e^{-2d}} >1.$ Hence, we have
\begin{eqnarray*}
e^{-2d} \frac{1}{1-e^{-2d}}< 1 & \Leftrightarrow & 2e^{-2d} < 1 \\
& \Leftrightarrow & -2d < -\ln 2 \\
& \Leftrightarrow & d >  \ln \sqrt 2.
\end{eqnarray*}

Now, for $\# \bnu = m >1$ we obtain
\begin{eqnarray*}
\sum_{{\bnu  \in\cI} \atop {\# \bnu =m}} e^{-2d \sum_k  3^{k-1} \nu_k} & = & \sum_{{\bnu  \in\cI} \atop {\# \bnu =m}} e^{-2d    \nu_1} e^{-2d 3\nu_2} \cdots e^{-2d   3^{m-1} \nu_m} \\
& \leq & \left(  \sum_{k=1}^{\infty} e^{ -2d k} \right)^m =  \left(  \frac{e^{-2d}}{1-e^{-2d}} \right)^m, \quad m=1,2, 3, \ldots
\end{eqnarray*} so that
\begin{eqnarray*}
\sum_{\bnu  \in\cI} c_{\bnu}^{-2d} & = & 1 + \sum_{m=1}^{\infty} \sum_{{\bnu  \in\cI} \atop {\# \bnu =m}} e^{-2d \sum_k  3^{k-1} \nu_k} \\
& \leq & 1 + \sum_{m=1}^{\infty} \left(  \frac{e^{-2d}}{1-e^{-2d}} \right)^m \\
& = & \frac{1-e^{-2d}}{1-2e^{-2d}}.
\end{eqnarray*}

Finally, we remark that although none of the Banach algebras $\cH_{-p}$ is commutative the second inequality in (\ref{vage_ineq}) holds with the same value of constant $C_{p-q}$. Indeed, due to the multiplication rules (\ref{Eq:2:008}) and (\ref{Eq:2:008c}) we have that $f_{\bnu} g_{\bmu}\sigma(\bnu, \bmu) \e^{\bnu + \bmu} = g_{\bmu} f_{\bnu} \sigma(\bmu, \bnu)  \e^{\bmu + \bnu},$ so that $\left\| fg \right\|_{\mathcal{H}_{-p}} = \left\| gf \right\|_{\mathcal{H}_{-p}}.$ \end{proof}

\begin{proposition}
The space $\mathfrak{S}_{-1}$ equipped with the product induced by the coefficients is a strong algebra.
\label{strong}
\end{proposition}

\begin{proof}
Let us start by endowing $\mathfrak{S}_{-1}$ with the inductive topology. From Theorem~\ref{ineq} we get that the product of the algebra is separately continuous in every space $\mathcal{H}_{-p}$ which is the same as continuity in the inductive topology. Additionally, the product in $\mathfrak{S}_{-1}$ inherits associativity from our ternary Grassmann algebra $\cG_3$. Therefore, $\mathfrak{S}_{-1}$ has a Banach algebra structure and, thus, we can consider it as the inductive limit of Banach spaces, which makes it a strong algebra. \end{proof}


For more details on this proof, see \cite{MR3404695}. This also shows that the inductive topology is equivalent to the strong topology in $\mathfrak{S}_{-1}$. Furthermore, the product will also be
associative in $\mathfrak{S}_{-1}$ and the multiplication is jointly continuous (\cite{2013arXiv1302.3372A}, p. 215, case (iv), and also \cite{Bourbaki81EVT}, IV.23, Proposition 4).

Then by \cite[Theorem 3.7]{MR3029153} we have  $\mathfrak{S}_{-1}$ being nuclear, and the dual of a perfect space.

\begin{corollary} \label{f-conv}
Suppose $n\in\mathbb{N}$ and $f\in\mathcal{H}_{-p}\subseteq\mathcal{H}_{-p-2}$. Then,
\[
\left\Vert f^n\right\Vert_{\mathcal{H}_{-p-2}} \leq C_2^{n-1} \left\Vert f\right\Vert_{\mathcal{H}_{-p}}^n,
\]
where $C_2>0$ is as in Theorem \ref{ineq}.
\end{corollary}

\begin{proof} We have for every $f\in\mathcal{H}_{-p}\subseteq\mathcal{H}_{-p-2}$ that  $\left\Vert f\right\Vert_{\mathcal{H}_{-p-2}} \leq \left\Vert f\right\Vert_{\mathcal{H}_{-p}}.$ By Theorem \ref{ineq},
\begin{eqnarray*}
\left\| f^n\right\|_{\mathcal{H}_{-p-2}} & \leq & C_2 \left\| f\right\|_{\mathcal{H}_{-p}} \left\| f^{n-1}\right\|_{\mathcal{H}_{-p-2}} \\
     & \leq & C_2^2 \left\| f\right\|_{\mathcal{H}_{-p}}^2  \left\| f^{n-2} \right\|_{\mathcal{H}_{-p-2}} \\
     & \leq & C_2^{n-1} \left\| f\right\|_{\mathcal{H}_{-p}}^n.
\end{eqnarray*}
\end{proof}

\begin{corollary}
Consider a power series
\begin{equation}
F(\lambda)=\sum_{n\in\mathbb{N}_0} \alpha_n \lambda^n
\label{exp-ps}
\end{equation}
absolutely convergent in the open disk with radius $R$, with $\alpha_n,\lambda\in\mathbb{C}$ and $\mathbb{N}_0=\mathbb{N}\cup\{0\}$.

For $f\in\mathcal{H}_{-p}$ we have that if
\begin{equation}
\left\| f\right\|_{\mathcal{H}_{-p}} < \frac{R}{C_2}
\label{cond}
\end{equation}
\label{ps}
then $F(f)$ converges in $\mathcal{H}_{-p-2}$.
\end{corollary}

\begin{proof}
By assumption, if $|\lambda|<R$, power series \eqref{exp-ps} converges absolutely, i.e.,
\[
\sum_{n\in\mathbb{N}_0} \left|\alpha_n \lambda^n\right|=\sum_{n\in\mathbb{N}_0} \left|\alpha_n\right| \left|\lambda^n\right|<\infty.
\]
Applying Corollary \ref{f-conv}, we obtain the absolute convergence of $F(f)$ in the space $\mathcal{H}_{-p-2}$ via
\begin{eqnarray*}
\sum_{n\in\mathbb{N}_0} \left\Vert\alpha_n f^n\right\Vert_{\mathcal{H}_{-p-2}} & = & \sum_{n\in\mathbb{N}_0} |\alpha_n|^2 \left\Vert f^n\right\Vert_{\mathcal{H}_{-p-2}} \\
     & \leq & \alpha_0+C_2^{-1} \sum_{n\in\mathbb{N}} |\alpha_n|^2 \left(C_2 \left\Vert f\right\Vert_{\mathcal{H}_{-p}}\right)^n.
\end{eqnarray*}
Thus, $F(f)$ converges absolutely in $\mathcal{H}_{-p-2}$ if
\[
C_2 \left\Vert f\right\Vert_{\mathcal{H}_{-p}} < R
\]
or, $\left\Vert f\right\Vert_{\mathcal{H}_{-p}} < \frac{R}{C_2}$.
\end{proof}

\begin{corollary}
Suppose $F(\lambda)$ is power series like in the previous corollary (Corollary \ref{ps}). Then, for all $f\in\mathfrak{S}_{-1}$ such that its scalar part $f_{\mathbf{0}}$ satisfies \eqref{cond} we have that $F(f)$ converges in $\mathfrak{S}_{-1}$
\label{body-red}
\end{corollary}

\begin{proof}
Suppose $f\in\mathfrak{S}_{-1}$, then we know that there is $q_0\in\mathbb{Z}$ with $f\in\mathcal{H}_{-q}$ for each $q\geq q_0$. Using Theorem~\ref{ps} we have that to ensure convergence of $F(f)$ we need $\left\Vert f\right\Vert_{\mathcal{H}_{-q}} < R/C_2$, which in general is not valid. But, due to Proposition \ref{limit-norm}, this condition becomes
\[
\left| f_{\mathbf{0}} \right|^2 < \frac{R}{C_2},
\]
which gives us the statement of the corollary.
\end{proof}

\begin{corollary}
Suppose $f\in\mathfrak{S}_{-1}$. Then, we have that $f$ is invertible if and only if we have for its scalar part $f_{\mathbf{0}}\neq 0$.
\end{corollary}

\begin{proof}
Let us assume that $g$ is the inverse of $f$ and its scalar part is denoted by $g_{\mathbf{0}}$, then we have $
fg=1$ implies $f_{\mathbf{0}} g_{\mathbf{0}} = 1$ and $f_{\mathbf{0}}\neq 0$.

To show the opposite direction we suppose $f_{\mathbf{0}}\neq 0$, or with a convenient normalization $f_{\mathbf{0}}=1$.  From Corollary \ref{body-red} we get that
\[
F(f) = \sum_{n\in\mathbb{N}_0} (1-f)^n
\]
converges when the scalar part of $1-f$ is less than $C_2^{-1}$. But, $(1-f)_B=0$ and we get that $g=F(f)\in\mathfrak{S}_{-1}$ and $g$ is the inverse of $f$.
\end{proof}

\subsection{Berezin integration}

We now look into a proper definition of path-integration in the sense of Berezin. Berezin integrals are used in superspace theory as linear maps from polynomials in anti-commuting variables to elements of Grassmann algebras.  As we aim to establish stochastic processes on generalised Grassmann algebras  it is necessary to construct a proper path integration on the arising infinite dimensional spaces of anti-commuting random variables. Such a path-integration theory was developed in  \cite{rogers1986}, \cite{MR1172996} 
as a Fermionic counterpart of the quantum Bosonic case.

In what follows we assume $f : \Omega \rightarrow \overline{\cG}^{(2)}_3,$ where $\Omega =\mathbb{R}$ or $\mathbb{C}$. We begin with the definition of the left multiplication operator  acting on functions with values in $\overline{\cG}^{(2)}_3$. For each $f \in \mathfrak{S}_{-1}$ we define the  left multiplication operator $M_f$ as
\begin{equation} \label{Eq3.Multiplication}
g\in \mathfrak{S}_{-1} \mapsto M_f g := fg \in \mathfrak{S}_{-1}.
\end{equation}
Recall that, since the inductive algebra $\mathfrak{S}_{-1}$ is a strong algebra we have that the multiplication is jointly continuous. Using the basis elements of $\cG_3$ we obtain
$$M_f g = \sum_{\bnu  +\bmu \in \cI } f_{\bnu} g_{\bmu} \sigma(\bnu, \bmu) \e^{\bnu+\bmu}:=  \sum_{\bnu  +\bmu \in \cI } f_{\bnu} g_{\bmu} M_{\bnu} \e^{\bmu}.
$$

Thus, the left multiplication operator $M$ is defined by its action on the basis elements $\e^{\bnu}$ of $\cG_3,$
\begin{equation} \label{Eq:3.M_Operator}
M_{\bnu} \mapsto M_{\bnu} \e^{\bmu} :=  \sigma(\bnu, \bmu) \e^{\bnu+\bmu}, \qquad \# (\bnu+ \bmu)< \infty,
\end{equation}
where $\sigma(\cdot, \cdot)$ is as in (\ref{Eq:2:008c}). We notice that for $f = \sum_{\bnu \in \cI} f_{\bnu} \e^{\bnu}$ we have
$$M_f 1 = \sum_{\bnu \in \cI} f_{\bnu} \sigma(\bnu, \mathbf{0})  \e^{\bnu} = f,$$
as $ \sigma(\bnu, \mathbf{0})=1.$

An obvious problem that arises is that the left multiplication has a non-trivial kernel. To overcome this we use the correspondent $\ell^2-$inner product linked to the $2$-norm of $\overline{\cG}^{(2)}_3,$
\begin{equation}
\inner{c_{\bnu}\e^{\bnu}, c_{\bta}\e^{\bta}}_2 :=   c_{\bnu} \overline c_{\bta} \delta_{\bnu,\bta},\quad c_{\bnu}, c_{\bta} \in \mathbb{C}, \quad \bnu, \bta \in \{ 0,1,2\}^{\BN}.
\end{equation}

Also we observe that $\e^{\mathbf{0}} = 1$ so that $M_{\mathbf{0}} = \Id$ is the identity operator, and $M_{\bnu} = M_{\nu_1} \cdots M_{\nu_d}$ for $\bnu = (\nu_1, \ldots, \nu_d).$

Then,
\begin{gather*}
\inner{ M_{\bnu} \e^{\bmu}, \e^{\bta}}_2 = \sigma(\bnu, \bmu) \inner{\e^{\bnu+\bmu}, \e^{\bta}}_2 = \sigma(\bnu, \bmu)\delta_{\bnu+\bmu, \bta},
\end{gather*} where $\bnu+\bmu \in \{ 0,1,2\}^{\BN}.$

For an arbitrary $\bnu \in \{ 0,1,2\}^{\BN}$ we define the adjoint of $M_{\bnu},$ denoted by  $M^\ast_{\bnu},$ as
\begin{eqnarray*} \label{Eq:3.M_operator}
\inner{M^\ast_{\bnu}  \e^{\bmu}, \e^{\bta}}_2 & := &  \inner{ \e^{\bmu}, M_{\bnu}  \e^{\bta}}_2 \\
& =&  \overline{\sigma(\bnu, \bta) } \inner{ \e^{\bmu}, \e^{\bnu+\bta}}_2  \\
& =&  \overline{\sigma(\bnu, \bta) } \delta_{\bmu, \bnu+\bta}  \\
& =&  \overline{\sigma(\bnu, \bta) } \delta_{\bmu - \bnu, \bta}
\end{eqnarray*} leading to
\begin{equation} \label{Eq:3.M_AdjointOperator}
M_{\bnu}^\ast \mapsto M^\ast_{\bnu}  \e^{\bmu} = \overline{\sigma(\bnu, \bmu - \bnu) } \e^{\bmu - \bnu}, \quad \bmu - \bnu \in \cI,
\end{equation} with again $M^\ast_{\mathbf{0}} = \Id$ and $M^\ast_{\bnu} = M^\ast_{\nu_d} M^\ast_{\nu_{d-1}} \cdots M^\ast_{\nu_1}$ for $\bnu = (\nu_1, \ldots, \nu_d).$  Furthermore, for $ f = \sum_{\bnu \in \cI} f_{\bnu} \e^{\bnu}, g = \sum_{\bmu \in \cI} g_{\bmu} \e^{\bmu}  \in \mathfrak{S}_{-1}$ we have $\inner{M^\ast_{f}  g, \e^{\bta}}_2 =  \inner{ g, M_{f}  \e^{\bta}}_2$ so that the adjoint becomes
$$M^\ast_{f}  g = \sum_{\bmu-\bnu \in \cI} \overline f_{\bnu} g_{\bmu} \overline{\sigma(\bnu, \bmu - \bnu) } \e^{\bmu - \bnu}.$$
Again, we notice that for $f = \sum_{\bnu \in \cI} f_{\bnu} \e^{\bnu}$ we have
$$M^\ast_{f} 1 = \sum_{\mathbf{0}-\bnu \in \cI} \overline f_{\bnu} \overline{\sigma(\bnu, - \bnu) } \e^{\mathbf{0} - \bnu} = \overline f_{\mathbf{0}},$$
as $\mathbf{0}-\bnu \in \cI$ if and only if $\bnu = \mathbf{0}$ and $\sigma(\mathbf{0}, \mathbf{0})=1.$

We finalize the description of the operators $M_{\bnu}  \e^{\bmu}$ and $M^\ast_{\bnu}  \e^{\bmu}$ with a table of the relevante pairs for each $j-$th component of $\bnu, \bmu \in \cI:$
\begin{center}
$\mu_j + \nu_j \in \{0,1,2 \}$  \hspace{1.1cm} $\mu_j - \nu_j \in \{0,1,2 \}$

~

\begin{tabular}{cc||c|c|c|}
    & $\mu_j$ & 0 & 1 & 2 \\
$\nu_j$  & $\diagdown$  &  &  &  \\
\hline
\hline
0  & & 0 & 1 & 2 \\
1  & & 1 & 2 &  \\
2  & & 2 &  &  \\
\hline
\end{tabular}    \hspace{1cm} \begin{tabular}{cc||c|c|c|}
    & $\mu_j$ & 0 & 1 & 2 \\
$\nu_j$  & $\diagdown$  &  &  &  \\
\hline
\hline
0  & & 0 & 1 & 2 \\
1  & &  & 0 & 1 \\
2  & &  &  & 0 \\
\hline
\end{tabular}
\end{center}

Remark that the $M^\ast_{\bnu}$ operator corresponds to a left derivative and it is analogous to the one traditionally defined in superanalysis and supersymmetry \cite{berezin1979method, MR914369, MR1172996, zbMATH00861741}. Hence, the Berezin integral can be defined in terms of $M^\ast_{\bnu}$ as
\begin{equation} \label{Eq:3_BerezinInt}
\int d \e^{\bnu} g := M^\ast_{\bnu} g = \sum_{\bmu-\bnu \in \cI} g_{\bmu} \overline{\sigma(\bnu, \bmu-\bnu)} \e^{\bmu-\bnu},
\end{equation} for $g = \sum_{\bmu \in \cI} g_{\bmu} \e^{\bmu} \in \mathfrak{S}_{-1}, \bnu \in \{ 0,1,2\}^{\BN}$ and where $\# \bnu < \infty.$

We remark that, in the particular case of $g = g_{\bnu} \e^{\bnu}$ then its Berezin integral becomes
$$\int d \e^{\bnu} (g_{\bnu} \e^{\bnu}) = g_{\bnu}  M^\ast_{\bnu} \e^{\bnu} = g_{\bnu} \overline{\sigma(\bnu, \mathbf{0})} \e^{\mathbf{0}} = g_{\bnu} = \inner{M_{g_{\bnu}\e^{\bnu}} 1, \e^{\bnu}}_2.$$

\begin{lemma}
Let $f,g \in \mathcal{H}_{0}.$ 
Then it holds:
\begin{equation}
\label{paris123}
\inner{ M_f 1, M_g 1}_2 = \inner{f,g}_2.
\end{equation}
\end{lemma}

The lemma is immediate since we have
$$M_f  1 = f,$$ as seen above.


\section{Stochastic Processes and Their Derivatives}
\label{sec-processes}

A second-order stochastic process indexed by a set $S$ is a map $f_t$ from $S$ into some probability space $\mathbf L^2(\Omega,\mathcal B,P)$, and the covariance of the process is
\begin{equation}
k(t,s)=\int_\Omega  \overline{f_t(w)}f_s(w)dP(w)\stackrel{\rm def.}{=}\mathbb E_P\overline{f_t}f_s,
\end{equation}
where $\mathbb E_P$ denotes the mathematical expectation with respect to $P$. Usually, in topological supersymmetry one is interested not just in the total probability distribution $P$ but also in the generalized probability distribution which consists the differential forms. For the sake of simplicity we are restricting us here to the case of $P$ with the consideration of the generalized probability distribution being done in a similar fashion than the classic case~\cite{Ovch11,Ovch16}.
In order to define stochastic integrals it is of interest to consider cases where the function $s\mapsto f_s$ is differentiable, possibly in a larger space
than the original probability space (a space of stochastic distributions).  Taking Hida's white noise space (see e.g. see \cite{MR1408433,MR1387829}) as probability space,
this space of stochastic distributions, together with an underlying space of stochastic test functions, form a Gel'fand triple. which allows to give useful models for
stochastic processes and their derivatives and in which one can develop stochastic calculus.\smallskip

There is more than one possible such Gel'fand triple. One particularly convenient space of stochastic distributions has been introduced by Yuri Kondratiev, and has a special algebraic structure. It is a strong algebra, as defined above, and in fact provided to the authors of \cite{vage1} the inspiration and framework to define strong algebras.\smallskip

Hida's white noise space
is identified in a natural way with the Fock space associated to $\ell^2(\mathbb N_0,\mathbb C)$, and this motivates the
definition of stochastic processes as functions (or, as multiplication operators by functions) taking valued in the counterpart of the Fock space in various situations.
This approach was also developed in \cite{aal2,aal3}, and in \cite{alpay2020generalized} in the setting of the grey noise space, and in \cite{APS2019} in the theory of non-commutative stochastic
processes and
\cite{MR3231624} in the setting of the Grassmann algebra. The Wick product takes different forms in each of these cases, but they all satisfy V\r{a}ge's inequality in an appropriately defined strong algebra;
this allows to  transfer the results from one setting to the other setting with the same proofs. We now introduce the counterpart of the Fock space in the present
framework.

\begin{definition} \label{defdef}
By analogy with the noncommutative setting, we define the $3-$graded super Fock space as $\overline{\cG}^{(2)}_3$.
\end{definition}

We here explain the corresponding theory in our setting, and first define what is meant by a stochastic process in the present framework.
Let $(\xi_n)_{n\in\mathbb N}$ denote the system of normalized Hermite functions. They form an orthonormal basis of $\mathbf L^2(\mathbb R,dx)$ and every element $f$ in the latter can thus be written as
\begin{equation}
f(u)=\sum_{n=1}^\infty f_n\xi_n(u),\quad {\rm with}\quad\sum_{n=1}^\infty |f_n|^2<\infty.
\end{equation}
We define an isometric map
\[
f\,\,\mapsto\,\, Xf=\sum_{n=1}^\infty f_n\mathbf e_n
\]
from $\mathbf L^2(\mathbb R,dx)$ into $\overline{\cG}^{(2)}_3$, and
\begin{equation}
M_{Xf}=\sum_{n=1}^\infty f_nM_{\mathbf e_n}
\end{equation}

\begin{definition}
A stochastic process indexed by a set $S$ is a map $s\mapsto M_{Xf_s}$, where $f_s\in\mathbf L^2(\mathbb R,dx)$ for every $s\in S$. The covariance function of the process is defined by
\begin{equation}
\langle M_{Xf_s}1, M_{Xf_t}1\rangle=\langle f_s,f_t\rangle_{2}.
\label{cov}
\end{equation}
\end{definition}
In \eqref{cov}, the second inner product is the $\mathbf L^2(\mathbb R,dx)$ inner product, and the equality follows from \eqref{paris123}. This equality plays a key role in
the arguments. Counterparts of this equality hold in particular in the white noise space setting, see \cite{new_sde}, the Poisson noise setting, see \cite[(4.9.4), p. 204]{new_sde},
the grey space noise setting, see \cite{J2015,alpay2020generalized}, and the free setting, see \cite{MR1217253}. In each case, the right and side stays the same, but the left hand side can take quite different forms. \eqref{cov} allows to relate the underlying setting with the Lebesgue space.\smallskip

In the free setting, the counterpart of the left hand side of \eqref{cov} is the trace of a $C^*$-algebra generated by the (real parts) of the creation operators. In the Grassmann setting, operators
are also involved, to make contact with the Berezin integral. Here too, to define counterparts of the Berezin integrals we introduced earlier multiplication operators, which can be seen as the analogs
of the creation operators. For the discussion of stochastic processes themselves, we will not consider operators, but directly functions.\smallskip

We are interested in two special cases, namely $S=\mathbb R$ (or a subinterval of it), and the real valued Schwartz functions, here denoted by $\mathcal S(\mathbb R)$.
In the first we consider covariance functions in \eqref{cov} of the form
\begin{equation}
K_\sigma(t,s) = \int_\mathbb{R} \frac{(e^{iut}-1)(e^{-ius}-1)}{u^2} d\sigma(u),
\label{kernel}
\end{equation}
where $\sigma$ represents an increasing function such that the Stieltjes integral
\begin{equation}
\label{qwertyu}
\int_\mathbb{R}\frac{d\sigma(u)}{u^2+1}<\infty.
\end{equation}
Such a family contains in particular the Brownian and the fractional Brownian motion.\smallskip

Let us introduce an operator $S_m$ in $\mathbf{L}^2(\mathbb{R})$ defined by
\begin{equation}
\widehat{S_mf}(u) = \sqrt{m(u)}\widehat{f}(u),
\label{Sm-def}
\end{equation}
with $\widehat{f}$ denoting the Fourier transform of $f$. Keep in mind that in general $S_m$ is an unbounded operator. The domain of $S_m$ is given by
\[
\text{dom}\,S_m = \left\{f\in\mathbf{L}^2(\mathbb{R})\suchthat \int_\mathbb{R} m(u) |\widehat{f}(u)|^2 du<\infty\right\},
\]
which contains $\mathbf{1}_{[0,t]}$. We can now consider the action of the operator $S_m$ on the function $\mathbf{1}_{[0,t]}$, i.e.
\[
f_m(t) = S_m\mathbf{1}_{[0,t]}
\]
and via an application of Plancherel's identity we get
\begin{eqnarray*}
\left\langle f_m(t),f_m(s)\right\rangle_{\mathbf{L}_2(\mathbb{R})} & = & \frac{1}{2\pi} \left\langle \widehat{f}_m(t),\widehat{f}_m(s)\right\rangle_{\mathbf{L}_2(\mathbb{R})} \\
     & = & \frac{1}{2\pi} \left\langle \sqrt{m(u)}\widehat{\mathbf{1}}_{[0,t]},\sqrt{m(u)}\widehat{\mathbf{1}}_{[0,s]}\right\rangle_{\mathbf{L}_2(\mathbb{R})} \\
     & = & \frac{1}{2\pi} \left\langle m(u)\frac{e^{-iut}-1}{u},\frac{e^{-ius}-1}{u}\right\rangle_{\mathbf{L}_2(\mathbb{R})} \\
     & = & \frac{1}{2\pi} \int_\mathbb{R} \frac{(e^{iut}-1)(e^{-ius}-1)}{u^2} m(u) du.
\end{eqnarray*}

The $\overline{\mathcal G}^{(2)}_3$-valued process
\[
  XS_m\mathbf{1}_{[0,t]}=\sum_{n\in\mathbb{N}} \left(\int_0^t (S_m\xi_n)(u) du\right){\mathbf e_n}
\]
has covariance function equal to
\[
\frac{1}{2\pi} \int_\mathbb{R} \frac{(e^{iut}-1)(e^{-ius}-1)}{u^2} m(u) du.
\]

\begin{theorem}
\label{thm55}
Let $m$ be a positive measurable function, satisfying \eqref{m-form}
\begin{equation}
m(u) \leq \left\{
\begin{array}{l l}
K |u|^{-b} & |u|\leq 1, \\
K |u|^{2N} & |u|> 1,
\end{array}
\right.
\label{m-form}
\end{equation}
with $b<2$, $N\in\mathbb{N}_0$, and $K$ represents a positive real constant.
and \eqref{qwertyu} (the latter for $d\sigma(t)=m(t)dt$).
Then, $TS_m1_{0,t]}$is differentiable in $\mathfrak{S}_{-1}$, with continuous derivative there.
\end{theorem}

\begin{theorem}
Let $s\mapsto f_s$ be a $\overline{\cG}^{(2)}_3$-valued function such that the derivative $s\mapsto f^\prime_s$ is continuous from $[0,1]$ into $\mathfrak S_{-1}$,
  and let
$s\mapsto Y(s)$ be a continuous function from $[0,1]$ into $\mathfrak S_{-1}$.  There is a $p\in\mathbb N$ such that the function $t\mapsto Y(t)f^\prime(t)$ is continuous
in $\mathcal H_{-p}$ and the corresponding Hilbert space integral $\int_0^1Y(t)f_s^\prime $ converges in $\mathfrak S_{-1}$.
\end{theorem}

In the case of $S=\mathcal S(\mathbb R)$, we consider a continuous positive operator $A$ from $\mathcal S(\mathcal R)$ into $\mathcal S^\prime(\mathbb R)$. On the one hand, applying the Bochner-Minlos theorem to the function $\exp(-\langle As,s\rangle)$  where the brackets denote the duality between $\mathcal S(\mathbb R)$ and
$\mathcal S^\prime(\mathbb R)$. We obtain a probability measure $P_A$ on $\mathcal S^\prime(\mathbb R)$ such that
\begin{equation}
\mathbb E_{P_A} e^{-i\langle \cdot, s\rangle}=e^{-\langle As,s\rangle},
\end{equation}
and a centred Gaussian process indexed by $\mathcal S(\mathbb R)$, defined by
\begin{equation}
Q_s(\ww)=\langle \ww, s\rangle,
\end{equation} with covariance function
\begin{equation}
\label{wa123}
\mathbb E_{P_A}(Q_{s_1}Q_{s_2})=\langle As_1,s_2\rangle.
\end{equation}
The operator $A$ can be factored via $\mathbf L^2(\mathbb R,dx)$ as $A=T^*T$, where $T$ is a continuous operator from
$\mathcal S(\mathbb R)$ into $\mathbf L^2(\mathbb R,dx)$ as $A=T^*T$, and therefore \eqref{wa123} can be rewritten as:
\begin{equation}
\mathbb E_{P_A}(Q_{s_1}Q_{s_2})=\langle Ts_1,Ts_2\rangle_2,
\end{equation}
where the latter brackets denote the inner product in $\mathbf L^2(\mathbb R,dx)$.

\newpage

\section*{Acknowledgement}
D. Alpay thanks the Foster G. and Mary McGaw Professorship in Mathematical Sciences, which supported this research. P. Cerejeiras and U. K\"ahler were supported by Portuguese funds through the CIDMA - Center for Research and Development in Mathematics and Applications, and the Portuguese Foundation for Science and Technology (``FCT--Funda\c{c}\~ao para a Ci\^encia e a Tecnologia''), within project UIDB/04106/2020 and UIDP/04106/2020. P. Cerejeiras was supported by the FCT  Sabbatical grant ref. SFRH/BSAB/143104/2018.

\bibliographystyle{plain}
\bibliography{all1}

\def\cprime{$'$} \def\cprime{$'$} \def\cprime{$'$}
  \def\lfhook#1{\setbox0=\hbox{#1}{\ooalign{\hidewidth
  \lower1.5ex\hbox{'}\hidewidth\crcr\unhbox0}}} \def\cprime{$'$}
  \def\cprime{$'$} \def\cprime{$'$} \def\cprime{$'$} \def\cprime{$'$}
  \def\cprime{$'$}
\begin{thebibliography}{10}

\bibitem{Abramov1997}
Viktor Abramov, Richard Kerner, and Bertrand Le~Roy.
\newblock Hypersymmetry: A z3-graded generalization of supersymmetry.
\newblock {\em Journal of Mathematical Physics}, 38(3):1650--1669, Mar 1997.

\bibitem{aal2}
D.~Alpay, H.~Attia, and D.~Levanony.
\newblock On the characteristics of a class of {G}aussian processes within the
  white noise space setting.
\newblock {\em Stochastic processes and applications}, 120:1074--1104, 2010.

\bibitem{aal3}
D.~Alpay, H.~Attia, and D.~Levanony.
\newblock White noise based stochastic calculus associated with a class of
  {G}aussian processes.
\newblock {\em Opuscula Mathematica}, 32/3:401--422, 2012.

\bibitem{alpay2020generalized}
D.~Alpay, P.~Cerejeiras, and U.~Kaehler.
\newblock Generalized fock space and moments, 2020.

\bibitem{MR3231624}
D.~Alpay, P.~Jorgensen, and G.~Salomon.
\newblock On free stochastic processes and their derivatives.
\newblock {\em Stochastic Process. Appl.}, 124(10):3392--3411, 2014.

\bibitem{alpay2018distribution}
D.~Alpay, I.L. Paiva, and D.C. Struppa.
\newblock Distribution spaces and a new construction of stochastic processes
  associated to the grassmann algebra.
\newblock {\em arXiv preprint arXiv:1806.11058}, 2018.

\bibitem{alpay2018positivity}
D.~Alpay, I.L. Paiva, and D.C. Struppa.
\newblock Positivity, rational schur functions, blaschke factors, and other
  related results in the grassmann algebra.
\newblock {\em arXiv preprint arXiv:1810.02843}, 2018.

\bibitem{vage1}
D.~Alpay and G.~Salomon.
\newblock New topological {$\Bbb C$}-algebras with applications in linear
  systems theory.
\newblock {\em Infin. Dimens. Anal. Quantum Probab. Relat. Top.},
  15(2):1250011, 30, 2012.

\bibitem{MR3038506}
D.~Alpay and G.~Salomon.
\newblock Non-commutative stochastic distributions and applications to linear
  systems theory.
\newblock {\em Stochastic Process. Appl.}, 123(6):2303--2322, 2013.

\bibitem{2013arXiv1302.3372A}
D.~{Alpay} and G.~{Salomon}.
\newblock {On algebras which are inductive limits of Banach spaces}.
\newblock {\em ArXiv e-prints}, February 2013.

\bibitem{MR3029153}
D.~Alpay and G.~Salomon.
\newblock Topological convolution algebras.
\newblock {\em J. Funct. Anal.}, 264(9):2224--2244, 2013.

\bibitem{MR3404695}
D.~Alpay and G.~Salomon.
\newblock On algebras which are inductive limits of {B}anach spaces.
\newblock {\em Integral Equations Operator Theory}, 83(2):211--229, 2015.

\bibitem{APS2019}
Daniel Alpay, Ismael~L. Paiva, and Daniele~C. Struppa.
\newblock Distribution spaces and a new construction of stochastic processes
  associated with the grassmann algebra.
\newblock {\em Journal of Mathematical Physics}, 60(1):013508, Jan 2019.

\bibitem{berezin1979method}
F.~A. Berezin.
\newblock {\em The Method of Second Quantization}.
\newblock Academic, New York, 1979.

\bibitem{MR914369}
F.A. Berezin.
\newblock {\em Introduction to superanalysis}, volume~9 of {\em Mathematical
  Physics and Applied Mathematics}.
\newblock D. Reidel Publishing Co., Dordrecht, 1987.
\newblock Edited and with a foreword by A. A. Kirillov, With an appendix by V.
  I. Ogievetsky, Translated from the Russian by J. Niederle and R. Koteck\'y,
  Translation edited by Dimitri Le\u\i tes.

\bibitem{Bourbaki81EVT}
N.~Bourbaki.
\newblock {\em Espaces vectoriels topologiques}.
\newblock Masson, 1981.

\bibitem{T94}
M.~R. de~Traubenberg.
\newblock Clifford algebras of polynomials generalized grassmann algebras and
  q-deformed heisenberg algebras.
\newblock {\em Adv. Appl. Clifford Algebras}, 4(CRN-94-17):131--144, 1994.

\bibitem{MR1172996}
B.~DeWitt.
\newblock {\em Supermanifolds}.
\newblock Cambridge Monographs on Mathematical Physics. Cambridge University
  Press, Cambridge, second edition, 1992.

\bibitem{MR0435834}
I.~M. Gelfand and N.~Ya. Vilenkin.
\newblock {\em Generalized functions. {V}ol. 4}.
\newblock Academic Press [Harcourt Brace Jovanovich Publishers], New York, 1964
  [1977].
\newblock Applications of harmonic analysis, Translated from the Russian by
  Amiel Feinstein.

\bibitem{GS2_english}
I.M. Gelfand and G.E. Shilov.
\newblock {\em {Generalized functions. Volume 2}}.
\newblock Academic Press, 1968.

\bibitem{Ali2}
A.~Guzmán~Adán and F.~Sommen.
\newblock Distributions and integration in superspace.
\newblock {\em J. Math. Phys.}, 59:073507, 2018.

\bibitem{Ali1}
A.~Guzmán~Adán and F.~Sommen.
\newblock Pizzetti and cauchy formulae for higher dimensional surfaces : a
  distributional approach.
\newblock {\em J. Math. Anal. Appl.}, 489:124140, 2020.

\bibitem{MR1408433}
H.~Holden, B.~{\O}ksendal, J.~Ub{\o}e, and T.~Zhang.
\newblock {\em Stochastic partial differential equations}.
\newblock Probability and its Applications. Birkh\"auser Boston Inc., Boston,
  MA, 1996.

\bibitem{new_sde}
Helge Holden, Bernt {\O}ksendal, Jan Ub{\o}e, and Tusheng Zhang.
\newblock {\em Stochastic partial differential equations}.
\newblock Universitext. Springer, New York, second edition, 2010.
\newblock A modeling, white noise functional approach.

\bibitem{J2015}
F.~Jahnert.
\newblock {\em Construction of a Mittag-Leffler Analysis and its Applications}.
\newblock PhD thesis, {TU Kaiserslautern}, 2015.

\bibitem{Junk1996}
Georg Junker.
\newblock {\em Supersymmetry in Classical Stochastic Dynamics}.
\newblock Theoretical and Mathematical Physics. Springer-Verlag Berlin
  Heidelberg, 1996.

\bibitem{K2010}
R.~Kerner.
\newblock Cubic and ternary algebras, ternary symmetries and the lorentz group.
\newblock {\em Proceedings of Math. Phys. Conference}, 1705:134--146, 2010.

\bibitem{Kerner2018}
R.~Kerner.
\newblock Ternary generalisations of graded algebras with some physical
  applications.
\newblock {\em Rev. Roumaine Math. Pures App.}, 489:107--141, 2018.

\bibitem{MR1387829}
H.-H. Kuo.
\newblock {\em White noise distribution theory}.
\newblock Probability and Stochastics Series. CRC Press, Boca Raton, FL, 1996.

\bibitem{Ovch11}
I.V. Ovchinnikov.
\newblock Self-organized criticality as witten-type topological field theory
  with spontaneous broken beeccchi-rouet-stora-tyutin symmetry.
\newblock {\em Physical Review E}, 83:051129, 2011.

\bibitem{Ovch16}
I.V. Ovchinnikov.
\newblock Introduction to supersymmetric theory of stochastics.
\newblock {\em Entropy}, 18(4):108, 2016.

\bibitem{PaSo1979}
G.~Parisi and N.~Sourlas.
\newblock Supersymmetric field theories and stochastic differential equations.
\newblock {\em Nucl. Phys.}, B206:321--332, 1979.

\bibitem{rogers1980global}
A.~Rogers.
\newblock A global theory of supermanifolds.
\newblock {\em Journal of Mathematical Physics}, 21(6):1352--1365, 1980.

\bibitem{rogers1986}
A.~Rogers.
\newblock Graded manifolds, supermanifolds and infinite-dimensional grassmann
  algebras.
\newblock {\em Commun.Math. Phys.}, (105):375--384, 1986.

\bibitem{zbMATH00861741}
A.~{Rogers}.
\newblock {\em {Supermanifolds. Theory and applications.}}
\newblock Singapore: World Scientific, 2007.

\bibitem{MR0209834}
L.~Schwartz.
\newblock {\em Th\'eorie des distributions}.
\newblock Publications de l'Institut de Math\'ematique de l'Universit\'e de
  Strasbourg, No. IX-X. Nouvelle \'edition, enti\'erement corrig\'ee, refondue
  et augment\'ee. Hermann, Paris, 1966.

\bibitem{vage96}
G.~V{\aa}ge.
\newblock Hilbert space methods applied to stochastic partial differential
  equations.
\newblock In H.~K{\"o}rezlioglu, B.~{\O}ksendal, and A.S. {\"U}st{\"u}nel,
  editors, {\em Stochastic analysis and related topics}, pages 281--294.
  Birk{\"a}user, {B}oston, 1996.

\bibitem{MR1217253}
D.~V. Voiculescu, K.~J. Dykema, and A.~Nica.
\newblock {\em Free random variables}, volume~1 of {\em CRM Monograph Series}.
\newblock American Mathematical Society, Providence, RI, 1992.
\newblock A noncommutative probability approach to free products with
  applications to random matrices, operator algebras and harmonic analysis on
  free groups.

\end{thebibliography}

\end{document}